\documentclass[runningheads]{llncs}

\usepackage[english]{babel}
\usepackage[utf8x]{inputenc}
\usepackage{graphicx}
\usepackage[colorinlistoftodos]{todonotes}
\usepackage{eufrak}
\usepackage{amssymb}
\usepackage{bbm}
\usepackage{tabto}
\usepackage{tikz-cd}
\usepackage{pifont}
\newcommand{\cmark}{\ding{51}}
\newcommand{\xmark}{\ding{55}}
\usepackage[ruled,vlined,linesnumbered]{algorithm2e}

\begin{document}
\title{Aggregating Relational Structures}
%
%
\author{Harshit Bisht\inst{1}\orcidID{0000-0003-4785-4989}
\and
Amit Kuber\inst{1}\orcidID{0000-0003-4812-1234}
}
\authorrunning{H. Bisht, A.S. Kuber}
%
\institute{Indian Institute of Technology, Kanpur, U.P. 208016, India}
\maketitle              
\begin{abstract}
We generalize the Arrow's impossibility theorem--a key result in social choice theory--to the setting where the arity $k$ of the relation under consideration is greater than $2$. Some special but natural properties of $k$-ary relations are considered, as well as an analogue for such $k$-ary relations of Endriss and Grandi's result on graph aggregation is proved.

\end{abstract}
\section{Introduction}
With the seminal result in \cite{arrow}, Arrow introduced the problem of aggregating voters' preferences into a collective preference ordering over a set of alternatives (candidates) and showed that such an aggregation must be dictatorial to satisfy seemingly reasonable constraints, laying the foundations for social choice theory. The result has been proved in multiple different settings (\cite{equivalence},\cite{poset}), and generalized to partially ordered preferences and graphs in \cite{poset} and \cite{graph} respectively. As computational tools make it easy to survey and assess data, summarisation and aggregation tools will be required to make inferences or decisions based on the observed relationships such as planning city routes given access to graphs corresponding to each pedestrian's activity. However, not all relationships can be adequately represented in the form of binary relations such as partial orders and graphs. Consider the relationship of ``knowing'' someone on a social network inferred from being tagged in the same picture, a group of three friends in the same photo carries more information than all combinations of pairs and can be best expressed by a 2-simplex. With the motivation of extending the problem to account for such complex relationships, we generalize the impossibility result to aggregating relational structures of any arity satisfying some special properties.

\section{Aggregating $k$-ary Relations}
Fix a non-empty finite set $A$ of candidates. Consider the predicate language $\mathcal{L}_A:=\{R,\{c_a\mid a\in A\}\}$ (for simplicity, we have assumed that $\mathcal L_A$ contains a single relation symbol $R$), where $R$ is a $k$-ary relation symbol with $k\geq 3$ and each $c_a$ is a constant symbol. Let $T_0$ consist of a single $\mathcal L_A$-sentence whose interpretation guarantees that each $\mathcal L_A$-structure has domain (in bijection with) $A$. We will deal with models of some $\mathcal L_A$-theory $T$ extending $T_0$; the extension $T$ will be specified in the due course. When the theory $T$ is fixed, in order to emphasize on the set $A$, we denote the collection of models of the theory $T$ by $\mathcal M(A)$.

A social choice situation over $\mathcal L_A$ consists of the following. Let $\mathcal I$ denote a set of voters/individuals. Each voter chooses $\mathcal{A}_i\in\mathcal M(A)$. An aggregation rule (also known as a social welfare function) is a map $\sigma:\mathcal D\subseteq\mathcal M(A)^{\mathcal{I}}\to\mathcal M(A)$ that satisfies some desirable properties, where $\mathcal D$ is the set of allowed ballots or profiles. By appropriately choosing $T$ and the properties of the aggregation rule, we will prove that the only legitimate choice of the aggregated structure is either a filtered product or an ultraproduct of $\{\mathcal{M}_i\}_{i\in\mathcal{I}}$; the latter case corresponds to a ``dictatorship'' when $\mathcal{I}$ is finite.

\subsection{Properties of $k$-ary relations}
Since Arrow's original result speaks of preference relations, it seemed natural to first attempt to extend it over relations. Owing at the lack of standard definitions of properties for $k$-ary relations, we attempt to provide our own and arrive at the result. We begin by setting up some notations.
\begin{itemize}
\item $(a_1,...,a_m)_{\delta^r}^{+{i_1},...,+{i_n},-{j_1},...,-{j_w}}$ refers to the set of subsequences of $(a_1,...,a_m)$ of length $m-r$, with elements $a_{i_1},...,a_{i_n}$ necessarily present and elements $a_{j_1},...,a_{j_w}$ necessarily absent.
\item $\bar{a}^{-j},b$ refers to the singleton element in $(a_1,...,a_j,b,a_{j+1},...,a_k)_{\delta}^{-j}$.
\end{itemize}

Over a set $A$, a $k$-ary relation $R$ is simply a subset of $A\times...\times A$ (k-times) or $A^k$. Now we define some desirable properties of $k$-ary relations that will make up the extended theory $T$.
\begin{definition}
A $k$-ary relation $R$ is called \textbf{connected} if, for each pairwise-distinct $a_1,...,a_k\in A$, there is a permutation $\tau$ of $\{1,...,k\}$ such that $(a_1,...,a_k)^\tau$ is in $R$.
\end{definition}
\begin{definition}
A $k$-ary relation $R$ is called \textbf{simplicial transitive} if for each $(k+1)$-ary sequence of pairwise-distinct elements $(a_1,...,a_k,a_{k+1})$, we have that  $(a_1,...,a_{k+1})^{+j}_{\delta}\subseteq R$ implies $(a_1,...,a_{k+1})^{-j}_{\delta}\subseteq R$ for each $j\in\{1,...,k+1\}$.
\end{definition}
\begin{definition}
A $k$-ary relation $R$ is called \textbf{path transitive} if for $k$-ary sequences of pairwise-distinct elements $(a_1,...,a_k)$ and $(b_1,...,b_k)$ with $a_i=b_j$ where $i>j$,  we have that $(a_1,...,a_{k})\in R$ and $(b_1,...,b_{k})\in R$ implies that $(a_1,...,a_{i-1},b_{j+1},...,b_{k})_{\delta^{i-j-1}}\subseteq R$.
\end{definition}
\begin{definition}
A $k$-ary relation $R$ is called \textbf{exclusive} if $R[\bar{a}^{\tau}]$ does not hold for all permutations $\tau$ of $1,...,k$ together.
\end{definition}

\subsection{Examples}
Several natural relationships are not adequately captured by binary relations. The following two examples satisfy the above defined properties and present natural aggregation scenarios where our result asserts that a desirable collective relation is impossible to produce.
\subsubsection{Moderate Voters}
Consider a collection of voters and a group of electoral candidates. Each voter interprets the political inclination of each candidate as left or right leaning compared to the others resulting in a total order over the set of candidates. If each voter prefers the moderate candidate in a group of 3 candidates, then this voting behaviour can be captured by a ``betweenness'' relation, with $(a,b,c)\in R_i\leftrightarrow(c,b,a)\in R_i$ representing the $i^{th}$ voters preference for $b$ over $a$ and $c$. Clearly, relations of this nature are both connected and exclusive. This relation is also simplicial transitive since for a sequence $a_1,...,a_4$ of candidates, specifying the betweenness of 3 restricted triples is enough to guarantee $a_1,...,a_4$ as strictly monotonous.
\subsubsection{Seating along a circular table}
Consider a party of dinner guests to be seated on circular table in groups of $4$ (any cyclic arrangement is fine) and preferences over how every subsets of 4 people should be seated. This quartenary relation $R$ on the party of dinner guests is such that if $(a,b,c,d)\in R$ then all cyclic permutations $(b,c,d,a),(c,d,a,b),(d,a,b,c)\in R$ and no other permutation of $\{a,b,c,d\}$ is in $R$. Again, this is both connected and exclusive. Now, consider the same voter's preferences over seating $5$ people on the same table. If he agrees on the cyclic permutation of every collection of 4 people but one, he will also agree to the final restriction of the original $5$-cycle. This is the definition of simplicial transitivity considered above. Thus, preferences on the seating of people around circular tables thus also satisfies the properties defined above.

Our main result implies that both these situations cannot be aggregated with certain desirable properties of the aggregation map.

\subsection{Properties of the aggregation map}
We considered some properties of $k$-ary relations above. Now we state the appropriate $k$-ary generalizations of the properties of the aggregation map that make the proof of the Arrow's impossibility theorem work.
\begin{itemize}
\item[(\textbf{UD})] $\forall a_1,...,a_{k+1}\in A.\forall p\in\mathcal{M}(\{a_1,...,a_{k+1}\})^{\mathcal{I}}.\exists q\in\mathcal{D}.q|_{a_1,...,a_k}=p$
\vspace{8pt}

Normally one expects the aggregation map to be defined on all $\mathcal I$-tuples of $T$-models, but the property \textbf{UD} generalizes the result by requiring the domain $\mathcal{D}$ to only be large enough for the proof to go through.
\vspace{8pt}
\item[(\textbf{P})] $\forall\bar{a}\in A^k.\forall p\in\mathcal{D}.(\forall i\in\mathcal{I}.p_{i}\models R[\bar{a}])\Rightarrow\sigma(p)\models R[\bar{a}]$
\vspace{8pt}

\textbf{P} requires that the result must satisfy the atomic formulas satisfied by each individual voter.
\vspace{8pt}
\item[(\textbf{IIA})] $\forall\bar{a}\in A^k.\forall p,q\in\mathcal{D}.(\forall i\in\mathcal{I}.p_i\models R[\bar{a}]\Leftrightarrow q_i\models R[\bar{a}])\Rightarrow(\sigma(p)\models R[\bar{a}]\Leftrightarrow\sigma(q)\models R[\bar{a}])$
\vspace{8pt}

\textbf{IIA} requires each atomic formula's satisfaction to be independent of other formulas.
\vspace{8pt}
\item[(\textbf{D})] $\exists i\in\mathcal{I}.\forall\bar{a}\in A^k.\forall p\in\mathcal{D}.(p_i\models R[\bar{a}]\Leftrightarrow\sigma(p)\models R[\bar{a}])$
\vspace{8pt}

\textbf{D} posits that the result satisfies the same atomic formulas as one specific voter.
\end{itemize}

\subsection{Arrow's theorem for a single $k$-ary relation}
Now we are ready to state and prove the generalization of Arrow's theorem to the situation for a single $k$-ary relation symbol. We deal with two sets of properties of $k$-ary relations, and prove the theorem in both cases parallelly.
\begin{theorem}
Let $(A,\mathcal{I},\mathcal{D},\sigma)$ be a social choice situation over $k$-ary relation $R$ language $\mathcal{L}$ with $|A|\geq k+1$, satisfying \textbf{UD},\textbf{P}, and \textbf{IIA} where $R$ is (simplicial or path) transitive, exclusive, and connected. Then for finite $\mathcal{I}$, it also satisfies \textbf{D}.
\end{theorem}

We prove this theorem in a series of propositions along the lines of \cite{abramsky}. First we introduce some notation. For any $U\subseteq\mathcal{I}$, define $U_{\bar{a}}:=\{p\in\mathcal{D}|p_U\models R[\bar{a}]\land p_{U^{c}}\not\models R[\bar{a}]\}$. Define a new $k$-ary relation $D_U$ on $A$ by $D_U[\bar{a}]:=\mathrm{distinct}(\bar{a})\land\forall p\in U_{\bar{a}}.\sigma(p)\models R[\bar{a}]$, where $\mathrm{distinct}(\bar{a})$ is short for $\bigwedge_{1\leq i<j\leq k}a_i\neq a_j$.
\begin{proposition}
For any tuple of pairwise distinct elements $\bar{a}$, $D_U[\bar{a}]\Rightarrow D_U[\bar{a}^{-j},b]$ for each $j\in\{1,...,k\}$.
\end{proposition}
\begin{proof}
If $a_j=b$, there is nothing to prove and we are done.\\
\textbf{If R is simplicial transitive:}\\
For $a_j\neq b$, we construct a profile p with $\bar{a}$, $(\bar{a}^{-j},b)$, and $(a_1,...,a_j,b,...,a_{k})_{\delta}^{j,j+1}$ holding (or not) at $U$ and $U^c$ according to the following table. The p as constructed can be posited to exist due to \textbf{UD}.
\begin{center}
\begin{tabular}{c|c c c}
 & $\bar{a}$ & $\bar{a}^{-j},b$ & $(a_1,...,a_j,b,...,a_{k})_{\delta}^{j,j+1}$ \\
 \hline
 $U$ & \cmark & \cmark & \cmark \\
 $U^c$ & \xmark & \xmark & \cmark \\
\end{tabular}
\end{center}
Clearly, $p\in U_{\bar{a}}\cap U_{\bar{a}^{-j},b}$, along with $p\in \mathcal{I}_{\bar{c}}$ for each $\bar{c}\in(a_1,...,a_j,b,...,a_{k})_{\delta}^{j,j+1}$. Thus, $\sigma(p)\models R[\bar{a}]$ (definition of $D_U$) and $\sigma(p)\models R[\bar{c}]$ for each $\bar{c}\in(a_1,...,a_j,b,...,a_{k})_{\delta}^{j,j+1}$ (\textbf{P}). By transitivity, we can now conclude $\sigma(p)\models R[\bar{a}^{-j},b]$.

Now, for any other profile $q\in U_{\bar{a}^{-j},b}$, we know that $\forall i\in\mathcal{I}.p_i\models R[\bar{a}^{-j},b]\Leftrightarrow q_i\models R[\bar{a}^{-j},b]$, making us conclude $\sigma(q)\models R[\bar{a}^{-j},b]$ (\textbf{IIA}) and thus that $D_U[\bar{a}^{-j},b]$ holds.\\
\textbf{If R is path transitive:}\\
For $a_j\neq b$, we construct a profile p with various combinations of $a_1,...,a_k,b$ holding (or not) at $U$ and $U^c$ according to the following table. The p as constructed can be posited to exist due to \textbf{UD}.
\begin{center}
\begin{tabular}{c|c c c}
 & $\bar{a}$ & $\bar{a}^{-j},b$ & $(a_2,...,a_j,b,...,a_k)$ \\
 \hline
 $U$ & \cmark & \cmark & \cmark \\
 $U^c$ & \xmark & \xmark & \cmark \\
\end{tabular}
\end{center}
Clearly, $p\in U_{\bar{a}}\cap U_{\bar{a}^{-j},b}$, along with $p\in \mathcal{I}_{\bar{c}}$ for $\bar{c}=(a_2,...,a_j,b,...,a_{k})$. Thus, $\sigma(p)\models R[\bar{a}]$ (definition of $D_U$) and $\sigma(p)\models R[\bar{c}]$ for $\bar{c}(a_2,...,a_j,b,...,a_{k})$ (\textbf{P}). By transitivity, we can now conclude $\sigma(p)\models\bar{a}^{-j},b$. Now, for any other profile $q\in U_{\bar{a}^{-j},b}$, we know that $\forall i\in\mathcal{I}.p_i\models R[\bar{a}^{-j},b]\Leftrightarrow q_i\models R[\bar{a}^{-j},b]$, making us conclude $\sigma(q)\models R[\bar{a}^{-j},b]$ (\textbf{IIA}) and thus that $D_U[\bar{a}^{-j},b]$ holds.\\
Note that this proof does not work for $j=1$. For that, the same proof with $(b,a_1,...,a_{k-1})$ (reversing position of $a_j$ and $b$) works.
\end{proof}
\begin{proposition}\label{prop2}
For any k-tuples of pairwise distinct elements $\bar{a}$ and $\bar{b}$, $D_U[\bar{a}]\Rightarrow D_U[\bar{b}]$.
\end{proposition}
\begin{proof}
We will work through the proof in two cases.\\
\textbf{Case 1:} Let $\bar{b}$ have the same elements as $\bar{a}$, only in a different permutation. Since $|A|\geq k+1$, we can get hold of $c\in A$ which is distinct from all elements in $\bar{a}$ (or $\bar{b}$). Consider the following inductive procedure:\\
\begin{algorithm}[H]
\DontPrintSemicolon
\SetAlgoLined
Beginning from $D[\bar{a}]$\\
Identify $a_m$ such that $b_1=a_m$\\
Conclude $D_U[\bar{a}^{-m},c]$ (by proposition 2)\\
Conclude $D_U[(\bar{a}^{-m},c)^{-1},b_1]$ (by proposition 2)\\
Conclude $D_U[((\bar{a}^{-m},c)^{-1},b_1)^{-m},a_1]$ (by proposition 2)\\
You have now shown $D_U$ holds for a permutation agreeing with $\bar{a}$ on the first element\\
Repeat similarly for $b_2,b_3,...,b_k$
\caption{Concluding $D_U[\bar{b}]$ from $D_U[\bar{a}]$}
\end{algorithm}
\textbf{Case 2:} Let $\bar{b}$ contain $c_1,...,c_n$ and $\bar{a}$ contain $d_1,...,d_n$ as the elements not shared between them. Conclude that $D_U$ holds at a sequence with each $d_i$ in $\bar{a}$ replaced arbitrarily by $c_i$'s using proposition 2. What is obtained now is a permutation of $\bar{b}$, allowing the conclusion to hold as shown in Case 1.
\end{proof}
The way we defined $D_U[\bar{a}]$ talks only about preference profiles where only the individuals in $U$ support $\bar{a}$. Intuitively, it is clear that profiles where even more individuals support $\bar{a}$ should also guarantee it to hold in the result. This is captured by the relation $$E_U[\bar{a}]:=\mathrm{distinct}(\bar{a})\land\forall p\in\mathcal{D}.(p_U\models R[\bar{a}]\Rightarrow\sigma(p)\models R[\bar{a}])$$.
\begin{proposition}
For any tuple of pairwise distinct elements $\bar{a}$, $D_U[\bar{a}]\Leftrightarrow E_U[\bar{a}].$
\end{proposition}
\begin{proof}
$E_U[\bar{a}]\Rightarrow D_U[\bar{a}]$ is clear from the definitions. We will thus focus on the proof of $D_U[\bar{a}]\Rightarrow E_U[\bar{a}]$.\\
\textbf{If R is simplicial transitive:}\\
Consider a profile $p$, not necessarily in $U_{\bar{a}}$ such that $p_U\models R[\bar{a}]$. We construct a profile $q$ with $\bar{a}$, $(b,a_2,a_3,...,a_k)$, and $(a_1,b,a_2,...,a_{k})_{\delta}^{1,2}$ holding (or not) at $U$ and $U^c$ according to the following table. The q as constructed can be posited to exist due to \textbf{UD}.
\begin{center}
\begin{tabular}{c|c c c}
 & $\bar{a}$ & $b,a_2,...,a_k$ & $(a_1,b,a_2,...,a_{k})_{\delta}^{1,2}$ \\
 \hline
 $U$ & \cmark & \cmark & \cmark \\
 $U^c$ & mimics p & \cmark & \xmark \\
\end{tabular}
\end{center}
Clearly, $q\in U_{\bar{c}}$ for each $\bar{c}\in(a_1,b,a_2,...,a_{k})_{\delta}^{1,2}$, along with $q\in \mathcal{I}_{b,a_2,...,a_k}$. Thus, $\sigma(q)\models R[b,a_2,...,a_k]$ (\textbf{P}) and $\sigma(q)\models R[\bar{c}]$ for each $\bar{c}\in(a_1,b,a_2,...,a_{k})_{\delta}^{1,2}$ ($D_U$ holds for each such $\bar{c}$). By transitivity, we can now conclude $\sigma(q)\models R[\bar{a}]$. It is now easy to see that $\forall i\in\mathcal{I}.p_i\models R[\bar{a}]\Leftrightarrow q_i\models R[\bar{a}]$, making us conclude $\sigma(p)\models R[\bar{a}]$ (\textbf{IIA}). Since the initial choice of $p$ was arbitrary, we can conclude $E_U[\bar{a}]$ holds.\\
\textbf{If R is path transitive:}\\
Consider a profile $p$, not necessarily in $U_{\bar{a}}$ such that $p_U\models R[\bar{a}]$. We construct a profile $q$ with $\bar{a}$, $\bar{a}^{-(k-1)},b$, and $\bar{a}^{-k},b$ holding (or not) at $U$ and $U^c$ according to the following table. The q as constructed can be posited to exist due to \textbf{UD}.
\begin{center}
\begin{tabular}{c|c c c}
 & $\bar{a}$ & $(a_1,...,a_{k-2},b,a_k)$ & $(a_1,a_2,...,a_{k-1},b)$ \\
 \hline
 $U$ & \cmark & \cmark & \cmark \\
 $U^c$ & mimics p & \xmark & \cmark \\
\end{tabular}
\end{center}
Again, $q\in U_{\bar{c}}$ for $\bar{c}=(a_1,...,a_{k-2},b,a_{k})$, along with $q\in \mathcal{I}_{a_1,...,a_{k-1},b}$. Thus, $\sigma(q)\models R[a_1,...,a_{k-1},b]$ (\textbf{P}) and $\sigma(q)\models R[\bar{c}]$ for $\bar{c}=(a_1,...,a_{k-2},b,a_{k})$ ($D_U$). By transitivity, we can now conclude $\sigma(q)\models R[\bar{a}]$. It is now easy to see that $\forall i\in\mathcal{I}.p_i\models R[\bar{a}]\Leftrightarrow q_i\models R[\bar{a}]$, making us conclude $\sigma(p)\models R[\bar{a}]$ (\textbf{IIA}). Since the initial choice of $p$ was arbitrary, we can conclude $E_U[\bar{a}]$ holds.
\end{proof}
Now, consider the following collection of voter coalitions $\mathcal{U}=\{U\subseteq\mathcal{I}\mid\exists\bar{a}\in A.D_U[\bar{a}]\}$. We wish to claim that this collection is an ultrafilter, which will allow us to find a "dictator" in the social choice situation, as described above.
\begin{claim}
$\mathcal{U}$ as defined above is an Ultrafilter.
\end{claim}
\begin{proof}
(\textbf{F1}: $\mathcal I\in\mathcal U$) Because of \textbf{P}, it is easy to note that $\mathcal{I}\in\mathcal{U}$.\\ \\
(\textbf{F2}: $\mathcal U$ is an upper set) Consider $U\in\mathcal{U}$ and $U\subseteq V$. Clearly, that gives us $E_U\subseteq E_V$, which allows us to conclude $V\in\mathcal{U}$ using Proposition 3.\\ \\
(\textbf{F3}: Any two elements of $\mathcal U$ intersect) Consider $U,V\in\mathcal{U}$ such that $U\cap V=\phi$. Construct a preference profile $p$ with all possible permutations of $\bar{a}$ holding (or not) according to the following table. The $p$ as constructed can be posited to exist due to \textbf{UD}
\begin{center}
\begin{tabular}{c|c c}
 & $\bar{a}$ & $\bar{a}^{\tau(\neq id)}$\\
 \hline
 $U$ & \cmark & \xmark \\
 $V$ & \xmark & \cmark \\
 $U^c\cap V^c$ & \xmark & \xmark \\
\end{tabular}
\end{center}
This would imply that $\sigma(p)\models R[\bar{a}^{\tau}]$ for each permutation $\tau$ (Since $U,V\in\mathcal{U}$), contradicting the \textbf{exclusivity} of $R$. Thus, such $U,V$ cannot belong in $\mathcal{U}$ together.\\ \\
(\textbf{F4}: $\mathcal U$ is a prime filter) Consider $U\in\mathcal{U}$ such that $U=W\sqcup V$. We must show one of $W$ or $V$ also belong in $\mathcal{U}$.\\
\textbf{If R is simplicial transitive:}\\
Construct a preference profile $p$ with various permutations of $\bar{a}$ and $a_1,...,a_k,b$ holding according to the following table:
\begin{center}
\begin{tabular}{c|c c c c c}
 & $\bar{a}$ & $\bar{a}^{\tau\neq id}$ & $(b,a_1,a_2,...,a_k)^{1,2}_{\delta}$ & $(b,a_1,a_2,...,a_k)_{\delta}$ & $(a_1,a_2,...,a_k,b)_{\delta}$\\
 \hline
 $V$ & \xmark & \cmark & \cmark & \xmark & \xmark\\
 $W$ & \cmark & \xmark & \xmark & \cmark & \xmark\\
 $U^c$ & \cmark & \xmark & \xmark & \xmark & \cmark\\
\end{tabular}
\end{center}
Since $R$ is connected, some permutation of $\bar{a}$ must hold at $\sigma(p)$. We evaluate two possible cases separately.\\
\textbf{Case 1:} Let $\sigma(p)\models R[\bar{a}^{\tau}]$ for some $\tau\neq id$. Clearly here, since $p\in V_{\bar{a}^{\tau}}$, we can conclude that $V\in\mathcal{U}$.\\ \\
\textbf{Case 2:} $\sigma(p)\models R[\bar{a}]$. Clearly, $p\in U_{\bar{c}}$ for each $\bar{c}\in(b,a_1,...,a_k)^{1,2}_{\delta}$ with $U\in\mathcal{U}$. Therefore, $\sigma(p)\models R[\bar{c}]$ for each $\bar{c}\in(b,a_1,...,a_k)^{1,2}_{\delta}$. By transitivity, we can conclude $\sigma(p)\models R[b,a_2,...,a_k]$. Since $p$ also belongs to $W_{b,a_2,...,a_k}$, we can conclude $W\in\mathcal{U}$.\\ \\
\textbf{If R is path transitive:}\\
Construct a preference profile $p$ with various combinations of $a_1,...,a_k,b$ holding according to the following table:
\begin{center}
\begin{tabular}{c|c c c c c}
 & $\bar{a}$ & $\bar{a}^{\tau\neq id}$ & $(a_2,...,a_k,b)$ & $(b,a_2,...,a_k)$ & $(a_1,...,a_{k-1},b)$\\
 \hline
 $V$ & \xmark & \cmark & \cmark & \xmark & \xmark\\
 $W$ & \cmark & \xmark & \cmark & \xmark & \cmark\\
 $U^c$ & \cmark & \xmark & \xmark & \cmark & \xmark\\
\end{tabular}
\end{center}
Since $R$ is connected, some permutation of $\bar{a}$ must hold at $\sigma(p)$. We evaluate two possible cases separately.\\
\textbf{Case 1:} Let $\sigma(p)\models R[\bar{a}^{\tau}]$ for some $\tau\neq id$. Clearly here, since $p\in V_{\bar{a}^{\tau}}$, we can conclude that $V\in\mathcal{U}$.\\ \\
\textbf{Case 2:} $\sigma(p)\models R[\bar{a}]$. Clearly, $p\in U_{\bar{c}}$ for $\bar{c}=(a_2,...,a_k,b)$ with $U\in\mathcal{U}$. Therefore, $\sigma(p)\models R[a_2,...,a_k,b]$. By transitivity, we can conclude $\sigma(p)\models R[a_1,...,a_{k-1},b]$. Since $p$ also belongs to $W_{a_1,...,a_{k-1},b}$, we can conclude $W\in\mathcal{U}$.
\end{proof}
This leads us to conclude that while aggregating $\mathcal{L}$-structures with single $k$-ary relation symbol $R$, \textbf{UD}, \textbf{P}, \textbf{IIA}, along with connectedness, exclusivity, and transitivity (any definition) of $R$ are sufficient conditions to conclude $\mathcal{U}$ is an ultrafilter. This is equivalent to the existence of a dictator while aggregating $k$-ary relations.

\section{Metaproperties for $k$-ary relations}
The above results have been established only for the properties we defined, restricting their usefulness. Borrowing the ideas put forward in \cite{graph}, we define metaproperties that collect a class of properties and show that an impossibility result follows for each of them.
\subsection{Defining metaproperties}
On a fixed set $V$, we will talk about $k$-ary relations $R\subseteq V^{k}$. Since such relations can be aptly described as uniform directed $k$-hypergraphs, we interchangeably call them $U_k$-graphs, for short. Any property $P$ of $k$-ary relations can be identified with the collection of all relations satisfying the property, i.e., a subset $P$ of $\mathcal{P}(V^k)$.

The social choice situation has $\mathcal{N}$ as the set of individuals, with each $i\in\mathcal{N}$ contributing a relation $R_{i}\subseteq V^k$ over a fixed set $V$. The collection of each voter's preference (a preference profile) $(R_{i})_{i\in\mathcal{N}}$ will be denoted by $\textbf{R}$. An aggregation rule $F:\mathcal{P}(V^k)^{\mathcal{N}}\rightarrow\mathcal{P}(V^k)$ takes a preference profile $\textbf{R}$ as input and outputs a collective relation $F(\textbf{R})$. In profile $\textbf{R}$, we will use $N_{\bar{a}}^{\textbf{R}}=\{i\in\mathcal{N}\mid \bar{a}\in R_i\}$ to denote the collection of voters in $\mathcal{N}$ supporting tuple $\bar{a}$ in $\textbf{R}$.
\begin{definition}
The \textbf{dictatorship} of an individual $i^*$ is the aggregation rule $F_{i^*}$ such that for each profile $\textbf{R}$, $F_{i^*}(\textbf{R})=R_{i^*}$.
\end{definition}
\begin{definition}
The \textbf{oligarchy} of a nonempty coalition $C^*$ is the aggregation rule $F_{C^*}$ such that for each profile $\textbf{R}$, $F_{i^*}(\textbf{R})=\bigcap_{i\in C^*}R_{i}$.
\end{definition}
We now restate the properties of aggregation rules with slight modifications to make it easier to work with them. \textbf{Unanimity} means that the aggregated $U_k$-graph will contain all tuples included in every voter's $U_k$-graph. \textbf{Groundedness} can be seen as unanimity with respect to abstinence/silence, stating that the result must contain a tuple only if at least one voter proposes it. Analogous to IIA, \textbf{independence of irrelevant edges} (IIE) requires that $F$ pays no attention to the other tuples of the $U_k$-graphs while making a decision about the inclusion about a particular tuple. Finally, an aggregation rule being \textbf{collectively rational} with respect to a property $P$ implies that its output satisfies $P$ whenever each relation in the profile does. This is especially useful when the result of the election must be in the same form as the votes (when aggregating total orders, for example). The preservation of properties preserve the relational structure of the inputs.

We wish to define suitable metaproperties that will both allow the proofs to go through and be general enough to be satisfied by a large class of relations. We wish for the properties to spread from a tuple in the $U_k$-graph to its ``neighbouring'' tuples (contagious), force the inclusion of certain tuples when some other tuples are present (implicative), and force atleast some tuples to exist so we forbid the empty relation on any induced $U_k$-subgraph (disjunctive). To make stating the results of this section easy, we will denote by $P[S^{+},S^{-}]$ the collection of all relations satisfying property $P$, containing all tuples in $S^{+}$ and none in $S^{-}$ for some disjoint $S^{+},S^{-}\subseteq V^k$.
\begin{definition}
Let $\bar{a},\bar{b}\in V^k$. A $U_k$-graph property $P$ is called $\bar{a}/\bar{b}$ contagious if there exist two disjoint sets $S^+,S^-\subseteq V^k$ such that:
\begin{enumerate}
\item For every $R\in P[S^+,S^-]$, $\bar{a}\in R$ implies $\bar{b}\in R$.
\item There exist $R_0,R_1\in P[S^+,S^-]$ with $\bar{a}\in R_1$ and $\bar{b}\notin R_0$.
\end{enumerate}
\end{definition}
\begin{definition}
A property $P$ is called \textbf{contagious} if it satisfies either of the conditions below:
\begin{enumerate}
\item For some j, $P$ is $\bar{a}/\bar{c}$ contagious for all distinct elements $a_1,...,a_k,b\in V$ for all $\bar{c}\in(a_1,...,a_{j-1},b,...,a_k)_{\delta}^{+j}$.
\item $P$ is $\bar{a}/\bar{c}$ contagious for all distinct elements $a_1,...,a_k,b\in V$ for all $j$ where $\bar{c}=\bar{a}^{-j},b$.
\end{enumerate}
\end{definition}

\begin{definition}
A property $P$ is called \textbf{implicative} if there exist two disjoint sets $S^+,S^-\subseteq V^k$ and three pairwise distinct tuples $\bar{a}_1,\bar{a}_2,\bar{a}_3\in V^k\backslash(S^+\cup S^-)$ such that:
\begin{enumerate}
\item For every relation $R\in P[S^+,S^-]$, $\bar{a}_1,\bar{a}_2\in R$ implies $\bar{a}_3\in R$.
\item There exist $U_k$-graphs $R_0,R_1,R_2,R_{13},R_{123}\in P[S^+,S^-]$ which satisfy\\ $R_0\cap\{\bar{a}_1,\bar{a}_2,\bar{a}_3\}=\phi$, $R_1\cap\{\bar{a}_1,\bar{a}_2,\bar{a}_3\}=\{\bar{a}_1\}$, $R_2\cap\{\bar{a}_1,\bar{a}_2,\bar{a}_3\}=\{\bar{a}_2\}$, $R_{13}\cap\{\bar{a}_1,\bar{a}_2,\bar{a}_3\}=\{\bar{a}_1,\bar{a}_3\}$, and $R_{123}\cap\{\bar{a}_1,\bar{a}_2,\bar{a}_3\}=\{\bar{a}_1,\bar{a}_2,\bar{a}_3\}$.
\end{enumerate}
\end{definition}
\begin{definition}
A $U_k$-graph property $P$ is called \textbf{disjunctive} if there exist two disjoint sets $S^+,S^-\subseteq V^k$ and two pairwise distinct tuples $\bar{a}_1,\bar{a}_2\in V^k\backslash(S^+\cup S^-)$ such that:
\begin{enumerate}
\item For every relation $R\in P[S^+,S^-]$, $\bar{a}_1\in R$ or $\bar{a}_2\in R$.
\item There exist relations $R_1,R_2\in P[S^+,S^-]$ with $R_1\cap\{\bar{a}_1,\bar{a}_2\}=\{\bar{a}_1\}$ and $R_2\cap\{\bar{a}_1,\bar{a}_2\}=\{\bar{a}_2\}$.
\end{enumerate}
\end{definition}

Before using the metaproperties defined to prove the required results, it is important to justify the choice of definition by verifying if they are satisfied by the properties used to complete the proof earlier. For example, to see that simplicial transitivity is contagious and implicative, we can choose $S^+=(a_1,...,a_j,b,...,a_k)_{\delta}^{+j,+(j+1)},$ $S^-=\phi$, and\\ $S^+=(a_1,...,a_j,b,...,a_k)_{\delta}^{+j,+(j+1)}\backslash\{(a_2,...,a_j,b,...,a_k)\},\ S^-=\phi$ respectively. In a similar spirit, to see that connectedness is disjunctive, choose $S^+=\{\bar{a}^\tau\mid\tau\neq\tau_1,\tau_2\}$ for some distinct permutations $\tau_1,\tau_2$ and $S^-=\phi$.

\subsection{Impossibility results}
For a tuple $\bar{a}$, consider the set $\mathcal{W}_{\bar{a}}$ such that $\bar{a}\in F(\textbf{R})\leftrightarrow N_{\bar{a}}^{\textbf{R}}\in\mathcal{W}_{\bar{a}}$. Thus, $\mathcal{W}_{\bar{a}}$ is the collection of winning coalitions for the tuple $\bar{a}$. An important property of elections is symmetry with respect to candidates. This would require that a coalition that could ensure the inclusion of one tuple is also able to ensure inclusion of all the other tuples or $\mathcal{W}_{\bar{a}}=\mathcal{W}$ for all tuples $\bar{a}$ and some collection of coalitions $\mathcal{W}$. We will only consider neutrality over the tuples with all distinct elements.

Now we are ready to prove $k$-ary analogues of some of the important results of \cite{graph}.
\begin{lemma}(Neutrality lemma)(cf. \cite[Lemma~12]{graph})
For $|V|\geq k+1$, any unanimous, grounded, and IIE aggregation rule that is collectively rational with respect to a contagious property must be neutral as defined above.
\end{lemma}
\begin{proof}
Consider first a property $P$ that is $\bar{a}/\bar{b}$ contagious for $\bar{a},\bar{b}\in V^k$. Consider an aggregation rule $F$ that is unanimous, grounded, IIE, and collectively rational with respect to $P$. We will show that $\mathcal{W}_{\bar{a}}\subseteq\mathcal{W}_{\bar{b}}$. Let $C$ be a coalition in $\mathcal{W}_{\bar{a}}$ and $S^{+},S^{-}\subseteq V^k$, $R_0,R_1\in P[S^+,S^-]$ be appropriate constructions from the $\bar{a}/\bar{b}$ contagiousness of $P$. Consider a profile $\textbf{R}$ where individuals in $C$ propose $R_1$ and others propose $R_0$. Since $C$ was a winning coalition for tuple $\bar{a}$, $\bar{a}\in F(\textbf{R})$. Also, $S^+\subseteq F(\textbf{R})$ and $S^-\cap F(\textbf{R})=\phi$ by unanimity and groundedness respectively. By collective rationality, $F(\textbf{R})\in P[S^+,S^-]$ and thus $\bar{b}\in F(\textbf{R})$. Since only voters in $C$ supported $\bar{b}$, $C\in\mathcal{W}_{\bar{b}}$.\\ \\
We know that all $\mathcal{W}_{\bar{a}}$'s are nonempty (since $\mathcal{N}\in\mathcal{W}_{\bar{a}}$ for each $\bar{a}$). Consider arbitrary pairwise distinct $\bar{a},\bar{b}\in V^k$. We will prove that $\mathcal{W}_{\bar{a}}\subseteq\mathcal{W}_{\bar{b}}$, which is sufficient to prove the claim. Notice the similarity to Proposition \ref{prop2}. Consider $C\in\mathcal{W}_{\bar{a}}$.
\begin{itemize}
\item If $P$ is contagious by condition $1$, we can use\\ $\bar{a}/[a_2,...,a_j,b_1,...,a_k]$-contagiousness to get $C\in\mathcal{W}_{[a_2,...,a_j,b_1,...,a_k]}$, followed by $[a_2,...,a_j,b_1,...,a_k]/[a_3,...,a_j,b_1,b_2,...,a_k]$-contagiousness\\for $C\in\mathcal{W}_{[a_3,...,a_j,b_1,b_2,...,a_k]}$, and so on till you get $C\in\mathcal{W}_{[b_1,...,b_j,a_{j+1},...,a_k]}$. Following that, apply contagiousness in a similar fashion but choose components of $\bar{b}$ in reverse ($b_k$ followed by $b_{k-1}$ and so on) while replacing. After $k$ steps, we can conclude $C\in\mathcal{W}_{\bar{b}}$.
\item If $P$ is contagious by condition $2$, we can use $\bar{a}/[b_1,a_2,...,a_k]$-contagiousness to get $C\in\mathcal{W}_{\bar{a}^{-1},b}$. Similarly, use the definition of contagiousness for values of $j$ increasing by $1$ till you conclude $C\in\mathcal{W}_{\bar{b}}$.
\end{itemize}
This completes the proof.
\end{proof}

Below is the filter-version of the main result.
\begin{theorem}(cf. \cite[Theorem~15]{graph})
(Oligarchy Theorem) For $|V|\geq k+1$, any unanimous, grounded, and IIE aggregation rule for $k$-ary relations that is collectively rational with respect to a contagious and implicative property must be oligarchic on pairwise distinct tuples.
\end{theorem}
\begin{proof}
Take any property $P$ which is contagious and implicative, along with an aggregation rule $F$ that is unanimous, grounded, IIE, and collectively rational with respect to $P$. As shown in the lemma above, we can now talk about a common collection of winning coalitions $\mathcal{W}$ such that $\bar{a}\in F(\textbf{R})\leftrightarrow N_{\bar{a}}^{\textbf{R}}\in\mathcal{W}$ for any pairwise distinct tuple $\bar{a}$. We will show that $\mathcal{W}$ is a filter, equivalent to claiming that $F$ is an oligarchic rule, with the oligarchy serving as the least element in the filter.\\ \\
Clearly, by unanimity $\mathcal{N}\in\mathcal{W}$.

To show $\mathcal{W}$ is closed under intersections, consider arbitrary $C_1,C_2\in\mathcal{W}$. Consider a profile $\textbf{R}$ where individuals in $C_1\cap C_2$ propose $R_{123}$, those in $C_1\backslash C_2$ propose $R_1$, those in $C_2\backslash C_1$ propose $R_2$, and others propose $R_0$ where the $U_k$-graphs are the ones from the definition of an implicative property. Since $C_1,C_2$ are winning coalitions, $\bar{a}_1,\bar{a}_2\in F(\textbf{R})$. Also, $S^+\subseteq F(\textbf{R})$ and $S^-\cap F(\textbf{R})=\phi$ by unanimity and groundedness respectively. By collective rationality, $F(\textbf{R})\in P[S^+,S^-]$ and thus $\bar{a}_3\in F(\textbf{R})$. Since only voters in $C_1\cap C_2$ supported $\bar{a}_3$, $C_1\cap C_2\in\mathcal{W}$.\\ \\
To show $\mathcal{W}$ is closed under upper bounds, consider arbitrary $C_1\in\mathcal{W}$ and $C_1\subseteq C_2$. Consider a profile $\textbf{R}$ where individuals in $C_1$ propose $R_{123}$, those in $C_2\backslash C_1$ propose $R_{13}$, and others propose $R_1$ where the $U_k$-graphs are the ones from the definition of an implicative property. Since $C_1$ is a winning coalitions, $\bar{a}_2\in F(\textbf{R})$. Also, $\bar{a}_1\in F(\textbf{R})$, $S^+\subseteq F(\textbf{R})$, and $S^-\cap F(\textbf{R})=\phi$ by unanimity and groundedness respectively. By collective rationality, $F(\textbf{R})\in P[S^+,S^-]$ and thus $\bar{a}_3\in F(\textbf{R})$. Since only voters in $C_2$ supported $\bar{a}_3$, $C_2\in\mathcal{W}$.\\ \\
Thus, we have successfully shown that $\mathcal{W}$ is a filter under the given assumptions.
\end{proof}
\begin{theorem}(cf. \cite[Theorem~16]{graph})
(Dictatorship Theorem) For $|V|\geq k+1$, any unanimous, grounded, and IIE aggregation rule for $k$-ary relations that is collectively rational with respect to a contagious, implicative, and disjunctive property must be dictatorial on pairwise distinct tuples.
\end{theorem}
\begin{proof}
Take any property $P$ which is contagious, implicative, and disjunctive, along with an aggregation rule $F$ that is unanimous, grounded, IIE, and collectively rational with respect to $P$. As shown in the above theorem, we can now talk about a common collection of winning coalitions $\mathcal{W}$ such that $\bar{a}\in F(\textbf{R})\leftrightarrow N_{\bar{a}}^{\textbf{R}}\in\mathcal{W}$ for any pairwise distinct $\bar{a}$ which is a filter. We will now show that it is an ultrafilter, equivalent to claiming that $F$ is a dictatorial rule, with the dictator serving as the least element in the ultrafilter.\\ \\
Consider arbitrary coalition $C$. Consider a profile $\textbf{R}$ where individuals in $C$ propose $R_{1}$, and others propose $R_2$ where the $U_k$-graphs are the ones from the definition of a disjunctive property. It follows  that $S^+\subseteq F(\textbf{R})$ and $S^-\cap F(\textbf{R})=\phi$ by unanimity and groundedness respectively. By collective rationality, $F(\textbf{R})\in P[S^+,S^-]$ and thus $\bar{a}_1\in F(\textbf{R})$ or $\bar{a}_2\in F(\textbf{R})$. Since only voters in $C$ supported $\bar{a}_1$ and only those in $\mathcal{N}\backslash C$ supported $\bar{a}_2$, $C\in\mathcal{W}$ or $\mathcal{N}\backslash C\in\mathcal{W}$. Thus, $\mathcal{W}$ is an ultrafilter.
\end{proof}
Thus, we have successfully extended the idea to $k$-ary relations.

\section{Discussion and future work}
\subsection{Some immediate consequences}
We have already seen that for social choice situations where our language carries a single relation symbol, the collection of winning coalitions forms an ultrafilter. Now, consider a language with many relation symbols $\{R^i\}_{i\in I}$. For each $R^i$, an argument similar to above would state that its interpretation is decided by an ultrafilter. Thus, the coalitions that decide all relations of the aggregate form an intersection of ultrafilters which is a filter.
\subsection{Future work}
\subsubsection{Model-theoretic structures}
In model theory, one often thinks of a filtered product of a collection of first-order structures, for a fixed language, as their average/aggregate. So far we have only dealt with a special class of relational structures. It will be very interesting to see whether it is possible to find some sufficient conditions that the interpretations of function symbols satisfy in order to force the aggregation to be a filtered product of those structures.

\subsubsection{Simplicial complexes}
A lot of applications model relationships as simplicial complexes with bounded dimension. Apart from aggregating social relationships, simplicial complexes are also useful in diverse areas including rendering $3D$-graphics. Such aggregation problems could arise naturally in decentralized computing setups when each unit produces a simplicial complex built on a predetermined grid of points.

%
%
%
%

\end{document}